\documentclass[11pt]{article}

\usepackage[T2A]{fontenc}
\usepackage[utf8]{inputenc}
\usepackage[english,russian]{babel}
\usepackage[tbtags]{amsmath}
\usepackage{amsfonts,amssymb}

\usepackage{indentfirst}

\usepackage{mathrsfs}

\usepackage{graphicx}
\usepackage{url}

\usepackage{pifont}

\newcommand{\circledOne}{\text{\ding{172}}}
\newcommand{\circledTwo}{\text{\ding{173}}}
\newcommand{\circledThree}{\text{\ding{174}}}

\newtheorem{theorem}{Теорема}
\newtheorem{lemma}{Лемма}
\newtheorem{proof}{Доказательство}

\newcommand{\RR}{\mathbb{R}}
\newcommand{\NN}{\mathbb{N}}
\newcommand{\PP}{\mathbb{P}}
\newcommand{\E}{\mathbb{E}}

\newenvironment{keywords} 
{
\list{}{\itemindent\listparindent\leftmargin 2\parindent
\rightmargin\leftmargin}%
\leftmargini 2em
\small\item[]\relax\parindent12pt\leftmarginii3em\hskip\parindent{\bf Ключевые слова:}}
{\endlist\vskip12pt plus 3pt minus 2pt}

\begin{document}

\def\udk#1{\def\@udk{#1}}
\udk{519.21}

\date{}

\author{Э.А.~ГОРБУНОВ (\textrm{ed-gorbunov@yandex.ru}),
	\\
	(Московский физико-технический институт, Москва),
	\\
Е.А.~ВОРОНЦОВА (\textrm{vorontsovaea@gmail.com}), \\
	(Дальневосточный федеральный университет, Владивосток),
	\\
	А.В.~ГАСНИКОВ (\textrm{gasnikov@yandex.ru}) , \\
	(Московский физико-технический институт, Москва)
	 }


\title{О верхней оценке математического ожидания нормы равномерно распределённого на сфере вектора и явлении концентрации
равномерной меры на сфере
\footnote{Работа А.В. Гасникова и Э.А. Горбунова поддержана грантом РФФИ 18-31-20005 мол\_а\_вед.
Работа Е.А. Воронцовой поддержана грантом РФФИ 18-29-03071.}}

\markboth{Э.\,А.~Горбунов, Е.\,А.~Воронцова, А.\,В.~Гасников}{О верхней оценке математического ожидания нормы равномерно распределённого на сфере вектора и явлении концентрации
равномерной меры на сфере}

\maketitle

\begin{abstract}
Рассматривается задача построения верхних оценок математического ожидания нормы равномерно распределённого 
на единичной евклидовой сфере вектора.  

Библиография: 15 названий.

\end{abstract}


\section{Введение}\label{sec1}
Пусть $e \in RS_2^n \left( 1 \right)$~--- случайный вектор, имеющий равномерное распределение на $n$-мерной единичной евклидовой сфере. 
	
В настоящей работе рассматривается задача построения уточнённых верхних оценок математического ожидания нормы вектора~$e$.
От точности построения таких оценок
зависят, например, оценки скорости сходимости
ускоренного метода  (Accelerated by Coupling Directional
Search~-- ACDS), построенного на базе специального каплинга спусков по направлению в форме градиентного спуска
и метода зеркального спуска~\cite{ACDS_VGG}.

Первые такие оценки были получены в 
2014 году в~\cite{Arxiv_GLUF},
статья опубликована в 2016 году в~\cite{AiT16_GLUF}.
Кроме того, независимо
от работы~\cite{AiT16_GLUF}
в 2015~г. похожие оценки были сделаны в~\cite{Shamir_arx},
опубликовано в 2017~г. в~\cite{Shamir_17};
и, также независимо, в 2015~г. одна
из возможных оценок математического ожидания нормы равномерно распределённого 
на единичной евклидовой сфере вектора была получена в~\cite{Duchi_15}.

Основным результатом данной работы является теорема~\ref{main_theorem}.

\section{Постановка задачи и формулировка результата}\label{sec2}	
	Пусть задан некоторый (неслучайный) вектор $s$ с единичной евклидовой сферы. Не умаляя общности, мы будем считать, что вектор $s$ направлен вдоль первой координатной оси (если это не так, то мы можем перейти к нужному базису). Тогда с вероятностью хотя бы $1 - \frac{2}{c}e^{-\frac{c^2}{2}}$ будет выполнено неравенство $|\langle s,\, e \rangle| \leqslant \frac{c}{\sqrt{n-1}}$ (см. теорему 2.7 и рисунок 2.2 из \cite{Blum_2016} и \cite{Ball_1997}). То есть, если взять $c=10$, то получим, что с большой вероятностью выполнено неравенство $\langle s,\, e \rangle^2\leqslant \frac{100}{n}$ (множество, на котором $\langle s,\, e \rangle^2\leqslant \frac{100}{n}$, обозначим через $A_s$; как мы видим, при достаточно больших $n$ вероятностная мера множества $A_s$ велика). Кроме того, можно показать, что $\E[\langle s,\, e \rangle^2] = \frac{1}{n}$ (см., например, лемму B.10 из \cite{bogolubsky2016learning}).
	
	Рассмотрим $\infty$-норму, которая для произвольного вектора $x \in \RR^n$ задаётся формулой $\|x\|_\infty = \underset{1\leqslant i\leqslant n}{\max}|x_i|$, где $x = (x_1,\,x_2,\,\ldots,\, x_n)^\top$. Заметим, что функция $f(e) = \|e\|_\infty$ является липшицевой с константой $1$ в евклидовой норме. Рассмотрим константу $M_f$ такую, что $\PP_{e}\left\{f(e) \geqslant M_f\right\} \geqslant \frac{1}{2}$ и $\PP_{e}\left\{f(e) \leqslant M_f\right\} \geqslant \frac{1}{2}$. Тогда верно неравенство (см. \cite{Boucheron_2013}, \cite{Milman_1986})
    $$
    \PP_e\left\{|f(e) - M_f| > t\right\} \leqslant 4e^{-\frac{t^2}{4}},\quad t > 0.
    $$
    Это означает, что случайная величина $\|e\|_\infty$ принимает очень близкие к $\E[\|e\|_\infty]$ ($M_f$ и $\E[f(e)]$ асимптотически близки, см. \cite{Zorich_2008}) значения на множестве достаточно большой меры. Кроме того, можно показать, что максимальная по модулю компонента вектора $e$ с вероятностью не меньше $1 - \frac{1}{n\sqrt{n}}$ принимает значения по модулю меньшие $\frac{2\sqrt{\ln n}}{\sqrt{n-1}}$ (множество, на котором $\|e\|_\infty \leqslant \frac{2\sqrt{\ln n}}{\sqrt{n-1}}$, обозначим через $B_\infty$). Тогда $\E[\langle s,\, e \rangle^2\|e\|_\infty^2]$ близко к среднему значению случайной величины $\langle s,\, e \rangle^2\|e\|_\infty^2$ на множестве $A_e \cap B_\infty$ (чья вероятностная мера по-прежнему велика), на котором она не превосходит ${400\ln n}/{n^2}$. Константа в этой оценке сильно завышена и она уточняется далее в Теореме~1 (причём не только для $\infty$-нормы). Однако такого рода рассуждения, вытекающие из явления концентрации равномерной меры на сфере, поясняют причины возникновения такой оценки, а также её целесообразность в терминах вхождения размерности пространства $n$.
	
	Сформулируем и докажем достаточно известный факт, заключающийся в том, что векторная $q$-норма является невозрастающей функцией от $q$ для любого фиксированного вектора.
	\begin{lemma}
	Для любого вектора $x\in\RR^n$ (и для любого $n \in \NN$) выполнено неравенство:
	\begin{equation}\label{eq:p_norm_estimation}
	  \|x\|_{p_1} \leqslant \|x\|_{p_2},  
	\end{equation}
	где $p_1 \geqslant p_2$ и под знаком $\|\cdot\|_q$ понимается векторная~$q$-норма (норма Гёльдера с показателем~$q$).
	\end{lemma}
	\begin{proof}
	Не умаляя общности, будем считать, что все компоненты вектора $x$ являются ненулевыми (если вектор $x$ ненулевой, то его норма равна норме его подвектора меньшей размерности, полученного удалением нулевых компонент, и, соответственно, можно рассматривать этот подвектор; если же вектор $x = 0$, то неравенство \eqref{eq:p_norm_estimation} тоже верно). Пусть $g_x(p) \overset{\text{def}}{=} \ln \|x\|_p = \ln\left(\sum\limits_{k=1}^{n}|x_k|^p\right)^{\frac{1}{p}} = \frac{1}{p}\ln\left(\sum\limits_{k=1}^{n}|x_k|^p\right)$. Тогда 
$$
\frac{dg_x(p)}{dp} = -\frac{1}{p^2}\ln\left(\sum\limits_{k=1}^{n}|x_k|^p\right) + \frac{1}{p}\cdot\frac{\sum\limits_{k=1}^{n}\ln(|x_k|)\cdot|x_k|^p}{\sum\limits_{k=1}^{n}|x_k|^p}.
$$
Так как $\ln y$~--- вогнутая по $y$ функция, то по неравенству Йенсена получаем
\begin{equation*}
    \begin{array}{cl}
        \frac{dg_x(p)}{dp}  &\leqslant \frac{1}{p}\ln\left(\sum\limits_{k=1}^{n}|x_k|^p\right)^{-\frac{1}{p}} + \frac{1}{p}\ln\left(\sum\limits_{k=1}^{n}|x_k|\cdot\frac{|x_k|^p}{\sum\limits_{k=1}^{n}|x_k|^p}\right)\\
        &= \frac{1}{p}\ln\left(\frac{\sum\limits_{k=1}^{n}|x_k|^{p+1}}{\left(\sum\limits_{k=1}^{n}|x_k|^{p}\right)^{\frac{p+1}{p}}}\right) \leqslant \frac{1}{p}\ln\left(\frac{\sum\limits_{k=1}^{n}|x_k|^{p+1}}{\sum\limits_{k=1}^{n}\left(|x_k|^{p}\right)^{\frac{p+1}{p}}}\right)\\
        & = 0,
    \end{array}
\end{equation*}
то есть функция $g_x(p)$~--- невозрастающая функция на $[1,+\infty)$. Лемма доказана.
	\end{proof}
	
	Имеет место следующая теорема, являющаяся следствием явления 
	концентрации равномерной меры на сфере вокруг экватора (см. также~\cite{BD}; северный полюс задается градиентом $\nabla f\left( x \right)$).
	
	\begin{theorem}\label{main_theorem}
		Пусть $e \in RS_2^n(1),\, n \geqslant8,\, s\in\RR^n$, тогда
		\begin{equation}\label{theorem1:expect_q_norm}
        \E[||e||_q^2] \leqslant \min\{q-1,\,16\ln n - 8\}n^{\frac{2}{q}-1},\quad 2\leqslant q \leqslant \infty
    \end{equation}
    
    \begin{equation}\label{theorem1:expect_inner_product}
        \E[\langle s,\, e\rangle^2||e||_q^2] \leqslant \sqrt{3}||s||_2^2\min\{2q-1,32\ln n -8\}n^{\frac{2}{q}-2},\quad 2\leqslant q \leqslant \infty,
    \end{equation}
    где под знаком $||\cdot||_q$ понимается векторная~$q$-норма (норма Гёльдера с показателем~$q$).
	\end{theorem}

	\begin{proof}
Докажем вспомогательное неравенство:
\begin{equation}\label{theorem1:rough_estimation_of_expectation_q_norm}
    \E[||e||_q^2] \leqslant (q-1)n^{\frac{2}{q}-1},\quad 2\leqslant q < \infty.
\end{equation}
Во-первых,
\begin{equation}\label{theorem1:jensen}
    \begin{array}{rl}
        \E[||e||_q^2] = \E\left[\left(\sum\limits_{k=1}^{n}|e_k|^q\right)^{\frac{2}{q}} \right] \overset{\circledOne}{\leqslant} \left(\E\left[\sum\limits_{k=1}^{n}|e_k|^q\right]\right)^{\frac{2}{q}} \overset{\circledTwo}{=} \left(n\E[|e_2|^q]\right)^{\frac{2}{q}},
    \end{array}
\end{equation}
где $\circledOne$ выполнено в силу вероятностного неравенства Йенсена (функция $\varphi(x) = x^{\frac{2}{q}}$ является вогнутой, так как $q\geqslant2$), а переход $\circledTwo$ корректен в силу линейности математического ожидания и одинаковой распределённости компонент вектора~$e$.

Во-вторых, воспользуемся тем фактом (лемма Пуанкаре, см., например, \cite[п.~6.3]{Poincare}), что
\begin{equation}
\label{eq_lem_puank}
    e \overset{d}{=} \frac{\xi}{\sqrt{\xi_1^2 + \dots + \xi_n^2}},
\end{equation}
где $\xi = (\xi_1,\xi_2,\ldots,\xi_n)^\top$~--- $n$-мерный гауссовский случайный вектор с нулевым математическим ожиданием и единичной ковариационной матрицей, а $\overset{d}{=}$ обозначает равенство по распределению. 
Тогда
\begin{equation*}
    \begin{array}{cl}
        \E[|e_2|^q] &= \E\left[\frac{|\xi_2|^q}{\left(\xi_1^2+\ldots+\xi_n^2\right)^{\frac{q}{2}}}\right]\\
        &= \idotsint\limits_{\RR^n}|x_2|^q\left(\sum\limits_{k=1}^{n}x_k^2\right)^{-\frac{q}{2}}\cdot\frac{1}{(2\pi)^{\frac{n}{2}}}\cdot \exp\left(-\frac{1}{2}\sum\limits_{k=1}^{n}x_k^2\right)dx_1\ldots dx_n.
    \end{array}
\end{equation*}
Перейдём к сферическим координатам:
\begin{equation*}
    \begin{array}{cl}
        x_1 &= r\cos\varphi\sin\theta_1\ldots\sin\theta_{n-2},\\
        x_2 &= r\sin\varphi\sin\theta_1\ldots\sin\theta_{n-2},\\
        x_3 &= r\cos\theta_1\sin\theta_2\ldots\sin\theta_{n-2},\\
        x_4 &= r\cos\theta_2\sin\theta_3\ldots\sin\theta_{n-2},\\
        \ldots\\
        x_n &= r\cos\theta_{n-2},\\
        &r>0,\,\varphi \in [0,2\pi),\, \theta_i \in [0,\pi],\, i = \overline{1,n-2}, 
    \end{array}
\end{equation*}
якобиан преобразования координат равен
\begin{equation*}
    \det\left(\frac{\partial(x_1,\ldots,x_n)}{\partial(r,\varphi,\theta_1,\theta_2,\ldots,\theta_{n-2})}\right) = r^{n-1}\sin\theta_1(\sin\theta_2)^2\ldots(\sin\theta_{n-2})^{n-2}.
\end{equation*}
Тогда математическое ожидание $\E[|e_2|^q]$ можно записать в виде:
\begin{equation*}
    \begin{array}{cl}
        \E[|e_2|^q] &= \idotsint\limits_{\substack{r>0,\,\varphi \in [0, \, 2\pi),\\ 
        \theta_i \in [0,\pi],\, i = \overline{1,n-2}}}r^{n-1}|\sin\varphi|^q|\sin\theta_1|^{q+1}|\sin\theta_2|^{q+2}\ldots|\sin\theta_{n-2}|^{q+n-2} \\
        &\cdot\;\frac{e^{-\frac{r^2}{2}}}{(2\pi)^{\frac{n}{2}}}dr\ldots d\theta_{n-2}\\
        &=\frac{1}{(2\pi)^{\frac{n}{2}}} I_r\cdot I_\varphi \cdot I_{\theta_1}\cdot I_{\theta_2}\cdot\ldots\cdot I_{\theta_{n-2}},
    \end{array}
\end{equation*}
где
\begin{equation*}
    \begin{array}{cl}
        I_r &= \int\limits_{0}^{+\infty}r^{n-1}e^{-\frac{r^2}{2}}dr,\\
        I_\varphi &= \int\limits_{0}^{2\pi}|\sin\varphi|^qd\varphi = 2\int\limits_{0}^{\pi}|\sin\varphi|^qd\varphi,\\
        I_{\theta_i} &= \int\limits_{0}^{\pi}|\sin\theta_i|^{q+i}d\theta_i,\, i=\overline{1, \, n-2}.
    \end{array}
\end{equation*}
Вычислим эти интегралы. Начнём с $I_r$:
\begin{equation*}
    \begin{array}{rl}
        I_r = \int\limits_{0}^{+\infty}r^{n-1}e^{-\frac{r^2}{2}}dr = /\text{замена}\;r = \sqrt{2t}/ = \int\limits_{0}^{+\infty}(2t)^{\frac{n}{2}-1}e^{-t}dt = 2^{\frac{n}{2}-1}\Gamma(\frac{n}{2}).
    \end{array}
\end{equation*}


Чтобы вычислить остальные интегралы, будет полезно рассмотреть следующий интеграл $(\alpha > 0)$:
\begin{equation*}
    \begin{array}{rl}
        \int\limits_{0}^{\pi} |\sin\varphi|^\alpha d\varphi &= 2\int\limits_{0}^{\frac{\pi}{2}}|\sin\varphi|^\alpha d\varphi = 2\int\limits_{0}^{\frac{\pi}{2}}(\sin^2\varphi)^{\frac{\alpha}{2}}d\varphi = /\text{замена}\; t = \sin^2\varphi /\\ 
        &= \int\limits_{0}^{1}t^{\frac{\alpha-1}{2}}(1-t)^{-\frac{1}{2}}dt = B(\frac{\alpha+1}{2},\,\frac{1}{2}) = \frac{\Gamma(\frac{\alpha+1}{2})\Gamma(\frac{1}{2})}{\Gamma(\frac{\alpha+2}{2})} = \sqrt{\pi} \frac{\Gamma(\frac{\alpha+1}{2})}{\Gamma(\frac{\alpha+2}{2})}.
    \end{array}
\end{equation*}
Отсюда получаем, что
\begin{equation}\label{theorem1:expectation_component}
    \begin{array}{cl}
        \E[|e_2|^q] &= \frac{1}{(2\pi)^{\frac{n}{2}}} I_r\cdot I_\varphi \cdot I_{\theta_1}\cdot I_{\theta_2}\cdot\ldots\cdot I_{\theta_{n-2}}\\
        &= \frac{2^{\frac{n}{2}-1}}{(2\pi)^{\frac{n}{2}}}\cdot\Gamma(\frac{n}{2})\cdot 2\frac{\sqrt{\pi}\Gamma(\frac{q+1}{2})}{\Gamma(\frac{q+2}{2})}\cdot\frac{\sqrt{\pi}\Gamma(\frac{q+2}{2})}{\Gamma(\frac{q+3}{2})}\cdot\frac{\sqrt{\pi}\Gamma(\frac{q+3}{2})}{\Gamma(\frac{q+4}{2})}\cdot\ldots\cdot\frac{\sqrt{\pi}\Gamma(\frac{q+n-1}{2})}{\Gamma(\frac{q+n}{2})}\\
        &= \frac{1}{\sqrt{\pi}}\cdot\frac{\Gamma(\frac{n}{2})\Gamma(\frac{q+1}{2})}{\Gamma(\frac{q+n}{2})}.
    \end{array}
\end{equation}
Покажем, что $\forall\, q\geqslant 2$
\begin{equation}\label{theorem1:key_estimation}
    \frac{1}{\sqrt{\pi}}\cdot\frac{\Gamma(\frac{n}{2})\Gamma(\frac{q+1}{2})}{\Gamma(\frac{q+n}{2})} \leqslant \left(\frac{q-1}{n}\right)^{\frac{q}{2}}.
\end{equation}
Сначала убедимся, что неравенство \eqref{theorem1:key_estimation} выполнено для $q=2$ (и произвольного $n$):
\begin{equation*}
    \frac{1}{\sqrt{\pi}}\cdot\frac{\Gamma(\frac{n}{2})\Gamma(\frac{2+1}{2})}{\Gamma(\frac{2+n}{2})} - \frac{1}{n} = \frac{1}{\sqrt{\pi}} \cdot \frac{\Gamma(\frac{n}{2})\cdot\frac{1}{2}\Gamma(\frac{1}{2})}{\frac{n}{2}\Gamma(\frac{n}{2})} - \frac{1}{n} = \frac{1}{n}-\frac{1}{n} = 0 \leqslant 0.
\end{equation*}
Рассмотрим функцию (вообще говоря, двух аргументов)
\begin{equation*}
    f_n(q) = \frac{1}{\sqrt{\pi}}\cdot\frac{\Gamma(\frac{n}{2})\Gamma(\frac{q+1}{2})}{\Gamma(\frac{q+n}{2})} - \left(\frac{q-1}{n}\right)^{\frac{q}{2}}
\end{equation*}
при $q \geqslant 2$. Также введём в рассмотрение функцию $\psi(x) = \frac{d(\ln(\Gamma(x)))}{dx}$ при $x > 0$ (\textit{дигамма-функция}). Для гамма-функции выполняется тождество
\begin{equation*}
    \Gamma(x+1) = x\Gamma(x),\, x>0.
\end{equation*}
Возьмём от этого тождества логарифм и продифференцируем по $x$:
\begin{equation*}
    \begin{array}{rl}
        \ln\Gamma(x+1) = \ln\Gamma(x) + \ln x,\\
        \frac{d(\ln(\Gamma(x+1)))}{dx} = \frac{d(\ln(\Gamma(x)))}{dx} + \frac{1}{x},
    \end{array}
\end{equation*}
что можно записать через дигамма-функцию:
\begin{equation}\label{theorem1:digamma_recurrence}
    \psi(x+1) = \psi(x) + \frac{1}{x}.
\end{equation}
Покажем, что дигамма-функция возрастает при $x>0$. Для этого докажем неравенство:
\begin{equation}
\label{theorem1:digamma_decrease}
    \left(\Gamma'(x)\right)^2 < \Gamma(x)\Gamma''(x).
\end{equation}
Действительно,
\begin{equation*}
    \begin{array}{cl}
        \left(\Gamma'(x)\right)^2 &= \left(\int\limits_0^{+\infty}e^{-t}\ln t\cdot t^{x-1}dt\right)^2\\
        &\overset{\circledOne}{<} \int\limits_0^{+\infty}\left(e^{-\frac{t}{2}}t^{\frac{x-1}{2}}\right)^2dt\cdot\int\limits_0^{+\infty}\left(e^{-\frac{t}{2}}t^{\frac{x-1}{2}}\ln t\right)^2dt\\
        &= \underbrace{\int\limits_0^{+\infty}e^{-t}t^{x-1}dt}_{\Gamma(x)}\cdot\underbrace{\int\limits_{0}^{+\infty}e^tt^{x-1}\ln^2 tdt}_{\Gamma''(x)},
    \end{array}
\end{equation*}
где $\circledOne$ следует из неравенства Коши-Буняковского (причём неравенство строгое, ибо функции $e^{-\frac{t}{2}}t^{\frac{x-1}{2}}$ и $e^{-\frac{t}{2}}t^{\frac{x-1}{2}}\ln t$ линейно независимы). Из неравенства \eqref{theorem1:digamma_decrease} следует, что
\begin{equation*}
    \frac{d^2(\ln\Gamma(x))}{dx^2} = \left(\frac{\Gamma'(x)}{\Gamma(x)}\right)' = \frac{\Gamma''(x)}{\Gamma(x)} - \frac{\left(\Gamma'(x)\right)^2}{\left(\Gamma(x)\right)^2} \overset{\eqref{theorem1:digamma_decrease}}{>} 0, 
\end{equation*}
то есть дигамма-функция возрастает.

Теперь покажем, что $f_n(q)$ убывает на отрезке $[2,+\infty)$. Для этого достаточно рассмотреть $\ln(f(q))$:
\begin{equation*}
    \begin{array}{cl}
        \ln(f_n(q)) &= \ln\left(\frac{\Gamma(\frac{n}{2})}{\sqrt{\pi}}\right) + \ln\left(\Gamma\left(\frac{q+1}{2}\right)\right) - \ln\left(\Gamma\left(\frac{q+n}{2}\right)\right) - \frac{q}{2}\left(\ln(q-1)-\ln n\right),\\
        \frac{d(\ln(f_n(q)))}{dq} &= \frac{1}{2}\psi\left(\frac{q+1}{2}\right)-\frac{1}{2}\psi\left(\frac{q+n}{2}\right)-\frac{1}{2}\ln(q-1)-\frac{q}{2(q-1)} + \frac{1}{2}\ln n.
    \end{array}
\end{equation*}
Покажем, что $\frac{d(\ln(f_n(q)))}{dq} < 0$ при $q\geqslant 2$. Пусть $k = \lfloor\frac{n}{2}\rfloor$ (ближайшее целое число, не превосходящее $\frac{n}{2}$). Тогда $\psi\left(\frac{q+n}{2}\right) > \psi\left(k-1+\frac{q+1}{2}\right)$ и $\ln n \leqslant \ln(2k+1)$, откуда следует, что
\begin{equation*}
    \begin{array}{cl}
        \frac{d(\ln(f_n(q)))}{dq} &< \frac{1}{2}\left(\psi\left(\frac{q+1}{2}\right)-\psi\left(k-1+\frac{q+1}{2}\right)\right)-\frac{1}{2}\ln(q-1)-\frac{q}{2(q-1)} + \frac{1}{2}\ln(2k+1)\\
        &\overset{\eqref{theorem1:digamma_recurrence}}{=}\frac{1}{2} \left(\psi\left(\frac{q+1}{2}\right) - \sum\limits_{i=1}^{k-1}\frac{1}{\frac{q+1}{2}+ k - i - 1} - \psi\left(\frac{q+1}{2}\right)\right) - \frac{q}{2(q-1)} + \frac{1}{2}\ln\left(\frac{2k+1}{q-1}\right)\\
        &\overset{\circledOne}{\leqslant} -\frac{1}{2}\sum\limits_{i=1}^{k-1}\frac{2}{q-1+2k-2i} - \frac{1}{q-1} + \frac{1}{2}\ln\left(\frac{2k+1}{q-1}\right)\\
        &= -\frac{1}{2}\left(\frac{2}{q-1} + \frac{2}{q+1} + \frac{2}{q+3}+ \ldots + \frac{2}{q+2k-3}\right) + \frac{1}{2}\ln\left(\frac{2k+1}{q-1}\right)\\
        &\overset{\circledTwo}{<} -\frac{1}{2}\ln\left(\frac{q+2k-1}{q-1}\right)+\frac{1}{2}\ln\left(\frac{2k+1}{q-1}\right) \overset{\circledThree}{\leqslant}-\frac{1}{2}\ln\left(\frac{2k+1}{q-1}\right)+\frac{1}{2}\ln\left(\frac{2k+1}{q-1}\right) = 0,
    \end{array}
\end{equation*}
где $\circledOne$ и $\circledThree$ выполнены в силу неравенства $q\geqslant2$, $\circledTwo$ следует из оценки сверху интеграла от функции $\frac{1}{x}$ интегралом от её верхней ступенчатой мажоранты $g(x) = \frac{1}{q-1+2i},\, x\in[q-1+2i,q-1+2i+2],\, i=\overline{0,2k-1}$: 
\begin{equation*}
    \frac{2}{q-1} + \frac{2}{q+1} + \frac{2}{q+3} + \ldots + \frac{2}{q+2k-3} > \int\limits_{q-1}^{q+2k-1}\frac{1}{x}dx = \ln\left(\frac{q+2k-1}{q-1}\right).
\end{equation*}

Итак, мы показали, что $\frac{d(\ln(f_n(q)))}{dq} < 0$ для $q\geqslant2$ и произвольного натурального $n$. Следовательно, для любого фиксированного $n$ функция $f_n(q)$ убывает по $q$, а значит, $f_n(q)\leqslant f_n(2) = 0$, то есть справедливо неравенство \eqref{theorem1:key_estimation}. Отсюда и из~\eqref{theorem1:jensen}, \eqref{theorem1:expectation_component} получаем, что
для любого
$q\geqslant2$
\begin{equation}\label{theorem1:pre-final}
    \E[||e||_q^2] \overset{\eqref{theorem1:jensen}}{\leqslant} \left(n\E[|e_2|^q]\right)^{\frac{2}{q}} \overset{\eqref{theorem1:expectation_component}, \, \eqref{theorem1:key_estimation}}{\leqslant} (q-1)n^{\frac{2}{q}-1}.
\end{equation}
Неравенство \eqref{theorem1:pre-final} нет смысла использовать при больших $q$ (относительно $n$). Рассмотрим правую часть неравенства \eqref{theorem1:pre-final} как функцию $q$ и найдём её минимум при $q\geqslant2$. Рассмотрим $h_n(q) = \ln(q-1) + \left(\frac{2}{q}-1\right)\ln n$ (логарифм правой части \eqref{theorem1:pre-final}). Производная $h(q)$:
\begin{equation*}
    \begin{array}{rl}
        \frac{dh(q)}{dq} &= \frac{1}{q-1} -\frac{2\ln n}{q^2},\\
        \frac{1}{q-1} -\frac{2\ln n}{q^2} &= 0,\\
        q^2-2q\ln n + 2\ln n &= 0.
    \end{array}
\end{equation*}
Если $n \geqslant 8$, то точка минимума на множестве $[2,+\infty)$ есть 
$$
q_0 = \ln n \left( 1+\sqrt{1-\frac{2}{\ln n}} \right )
$$ 
(в случае $n\leqslant 7$ оказывается, что $q_0 = 2$; везде далее считаем, что $n\geqslant8$). Поэтому для всех $q>q_0$ более точная оценка будет следующей:
\begin{equation}\label{theorem1:pre_final_big_q}
    \begin{array}{rl}
        \E[||e||_q^2] &\overset{\circledOne}{<} \E[||e||_{q_0}^2] \overset{\eqref{theorem1:pre-final}}{\leqslant}(q_0-1)n^{\frac{2}{q_0}-1}\overset{\circledTwo}{\leqslant}(2\ln n -1)n^{\frac{2}{\ln n}-1}\\
        &= (2\ln n -1)e^2\frac{1}{n}\leqslant(16\ln n -8)\frac{1}{n}\leqslant(16\ln n -8)n^{\frac{2}{q}-1},
    \end{array}
\end{equation}
где $\circledOne$ верно в силу Леммы~1, $\circledTwo$ следует из неравенств $q_0 \leqslant 2\ln n,\, q_0 \geqslant \ln n$. Объединяя оценки \eqref{theorem1:pre-final} и \eqref{theorem1:pre_final_big_q}, получаем неравенство \eqref{theorem1:expect_q_norm}. 

Теперь перейдём к доказательству неравенства \eqref{theorem1:expect_inner_product}. Во-первых, получим оценку для $\sqrt{\E[||e||_q^4]}$. В силу вероятностного неравенства Йенсена ($q \geqslant 2$)
\begin{equation*}
    \begin{array}{rl}
        \E[||e||_q^4] &= \E\left[\left(\left(\sum\limits_{k=1}^{n}|e_k|^q\right)^2\right)^{\frac{2}{q}}\right] \leqslant \left(\E\left[\left(\sum\limits_{k=1}^{n}|e_k|^q\right)^2\right]\right)^{\frac{2}{q}}\\
        &\overset{\circledOne}{\leqslant} \left(\E\left[\left(n\sum\limits_{k=1}^{n}|e_k|^{2q}\right)\right]\right)^{\frac{2}{q}} \overset{\circledTwo}{=} \left(n^2\E[|e_2|^{2q}]\right)^{\frac{2}{q}}\\
        &\overset{\eqref{theorem1:expectation_component},\eqref{theorem1:key_estimation}}{\leqslant} n^{\frac{4}{q}}\left(\left(\frac{2q-1}{n}\right)^{\frac{2q}{2}}\right)^{\frac{2}{q}} = (2q-1)^{2}n^{\frac{4}{q}-2},
    \end{array}
\end{equation*}
где $\circledOne$ следует из неравенства $\left(\sum\limits_{k=1}^{n} x_k\right)^2 \leqslant n\sum\limits_{k=1}^{n} x_k^2$ для $x_1,x_2,\ldots,x_n\in\RR$, а $\circledTwo$ есть следствие линейности математического ожидания и одинаковой распределённости компонент вектора $e$. Отсюда получаем оценку
\begin{equation}\label{theorem1:pre_final_q_norm_power4}
    \sqrt{\E[||e||_q^4]} \leqslant (2q-1)n^{\frac{2}{q}-1}.
\end{equation}
Рассмотрим правую часть неравенства \eqref{theorem1:pre_final_q_norm_power4} как функцию $q$ и найдём её минимум при $q\geqslant2$. Рассмотрим $h_n(q) = \ln(2q-1) + \left(\frac{2}{q}-1\right)\ln n$ (логарифм правой части \eqref{theorem1:pre_final_q_norm_power4}). Производная $h(q)$:
\begin{equation*}
    \begin{array}{rl}
        \frac{dh(q)}{dq} &= \frac{2}{2q-1} -\frac{2\ln n}{q^2},\\
        \frac{2}{2q-1} -\frac{2\ln n}{q^2} &= 0,\\
        q^2-2q\ln n + \ln n &= 0.
    \end{array}
\end{equation*}
Если $n \geqslant 3$, то точка минимума на множестве $[2,+\infty)$ есть 
$$
q_0 = \ln n\left(1+\sqrt{1-\frac{1}{\ln n}}\right)
$$
(в случае $n \leqslant 2$ оказывается, что $q_0 = 2$; везде далее считаем, что $n \geqslant 3$). Поэтому для всех $q>q_0$ более точная оценка будет следующей:
\begin{equation}\label{theorem1:pre_final_big_q_power4}
    \begin{array}{rl}
        \sqrt{\E[||e||_q^4]} &\overset{\circledOne}{<} \sqrt{\E[||e||_{q_0}^4]} \overset{\eqref{theorem1:pre_final_q_norm_power4}}{\leqslant}(2q_0-1) n^{\frac{2}{q_0}-1}\overset{\circledTwo}{\leqslant}(4\ln n -1)n^{\frac{2}{\ln n}-1}\\
        &= (4\ln n -1)e^2\frac{1}{n}\leqslant(32\ln n -8)\frac{1}{n}\leqslant(32\ln n -8)n^{\frac{2}{q}-1},
    \end{array}
\end{equation}
где $\circledOne$ верно в силу неравенства $||e||_q < ||e||_{q_0}$ для $q > q_0$, $\circledTwo$ следует из неравенств $q_0 \leqslant 2\ln n,\, q_0 \geqslant \ln n$. Объединяя оценки \eqref{theorem1:pre_final_q_norm_power4} и \eqref{theorem1:pre_final_big_q_power4}, получаем неравенство
\begin{equation}\label{theorem1:expect_q_norm_power4}
    \sqrt{\E[||e||_q^4]}\leqslant \min\{2q-1,32\ln n -8\}n^{\frac{2}{q}-1}.
\end{equation}

Теперь найдём $\E[\langle s,\,e\rangle^4]$, где $s\in\RR^n$~--- некоторый вектор. Пусть $S_n(r)$~--- площадь поверхности $n$-мерной евклидовой сферы радиуса $n$, $d\sigma(e)$~--- ненормированная равномерная мера на $n$-мерной евклидовой сфере. В данных обозначениях $S_n(r) = S_n(1)r^{n-1},\, \frac{S_{n-1}(1)}{S_n(1)} = \frac{n-1}{n\sqrt{\pi}}\frac{\Gamma(\frac{n+2}{2})}{\Gamma(\frac{n+1}{2})}$. Кроме того, пусть $\varphi$~--- угол между $s$ и $e$.
Тогда
\begin{equation}\label{theorem1:inner_product_power4}
    \begin{array}{rl}
        \E[\langle s,\, e\rangle^4] &= \frac{1}{S_n(1)}\int\limits_{S}\langle s,\, e\rangle^4d\sigma(\varphi) = \frac{1}{S_n(1)}\int\limits_0^\pi||s||_2^4\cos^3\varphi S_{n-1}(\sin\varphi)d\varphi\\
        &= ||s||_2^4\frac{S_{n-1}(1)}{S_n(1)}\int\limits_{0}^\pi\cos^4\varphi\sin^{n-2}\varphi d\varphi\\
        &= ||s||_2^4\cdot\frac{n-1}{n\sqrt{\pi}}\frac{\Gamma(\frac{n+2}{2})}{\Gamma(\frac{n+1}{2})}\int\limits_{0}^\pi\cos^4\varphi\sin^{n-2}\varphi d\varphi.
    \end{array}
\end{equation}
Отдельно вычислим интеграл:
\begin{equation*}
    \begin{array}{rl}
        \int\limits_0^\pi\cos^4\varphi\sin^{n-2}\varphi d\varphi &= 2\int\limits_0^{\frac{\pi}{2}}\cos^4\varphi\sin^{n-2}\varphi d\varphi = /\text{замена}\; t=\sin^2\varphi/\\
        &= \int\limits_0^{\frac{\pi}{2}}t^{\frac{n-3}{2}}(1-t)^{\frac{3}{2}}dt = B(\frac{n-1}{2},\frac{5}{2}) = \frac{\Gamma(\frac{5}{2})\Gamma(\frac{n-1}{2})}{\Gamma(\frac{n+4}{2})}\\
        &= \frac{\frac{3}{2}\cdot\frac{1}{2}\Gamma(\frac{1}{2})\Gamma(\frac{n-1}{2})}{\frac{n+2}{2}\cdot\Gamma(\frac{n+2}{2})} = \frac{3}{n+2}\cdot\frac{\sqrt{\pi}\Gamma(\frac{n-1}{2})}{2\Gamma(\frac{n+2}{2})}.
    \end{array}
\end{equation*}
Отсюда и из \eqref{theorem1:inner_product_power4} получаем, что
\begin{equation}\label{theorem1:inner_product_power4_final}
    \begin{array}{rl}
        \E[\langle s,\, e\rangle^4] &= ||s||_2^4\cdot\frac{n-1}{n\sqrt{\pi}}\frac{\Gamma(\frac{n+2}{2})}{\Gamma(\frac{n+1}{2})}\cdot\frac{3}{n+2}\cdot\frac{\sqrt{\pi}\Gamma(\frac{n-1}{2})}{2\Gamma(\frac{n+2}{2})}\\
        &= ||s||_2^4\cdot\frac{3(n-1)}{2n(n+2)}\cdot\frac{\Gamma(\frac{n-1}{2})}{\frac{n-1}{2}\Gamma(\frac{n-1}{2})}
        = \frac{3||s||_2^4}{n(n+2)} \overset{\circledOne}{\leqslant} \frac{3||s||_2^4}{n^2}.
    \end{array}
\end{equation}

Чтобы доказать неравенство \eqref{theorem1:expect_inner_product}, осталось воспользоваться \eqref{theorem1:expect_q_norm_power4}, \eqref{theorem1:inner_product_power4_final} и неравенством Коши-Буняковского ($(\E[XY])^2\leqslant\E[X^2]\cdot\E[Y^2]$):
\begin{equation*}
    \begin{array}{rl}
        \E[\langle s,\,e\rangle^2 ||e||_q^2] \overset{\circledOne}{\leqslant} \sqrt{\E[\langle s,\,e\rangle^4]\cdot\E[||e||_q^4]} \leqslant \sqrt{3}||s||_2^2\min\{2q-1,32\ln n -8\}n^{\frac{2}{q}-2}.
    \end{array}
\end{equation*}

Теорема доказана.	
\end{proof}

\section{Вычислительные эксперименты}
Для уточнения констант в верхних оценках теоремы~\ref{main_theorem} были проведены
вычислительные эксперименты.
При генерации случайных
векторов, равномерно распределённых на поверхности $n$-мерной
евклидовой сферы, использовалась лемма Пуанкаре (см.~\eqref{eq_lem_puank}) о том, что 
компоненты любого вектора~$e$, имеющего
такое распределение, можно представлять как отношения 
$\frac{e_k}{
\sqrt{e_1^2 + ... + e_n^2 }}$, где все $e_k$, $k = 1, \, 2, \, \ldots$~--- независимые одинаково распределённые случайные величины,
имеющие стандартное нормальное распределение~$N(0, \, 1)$.

\begin{figure}[ht]
			\center{\includegraphics[width=0.8
				\linewidth]{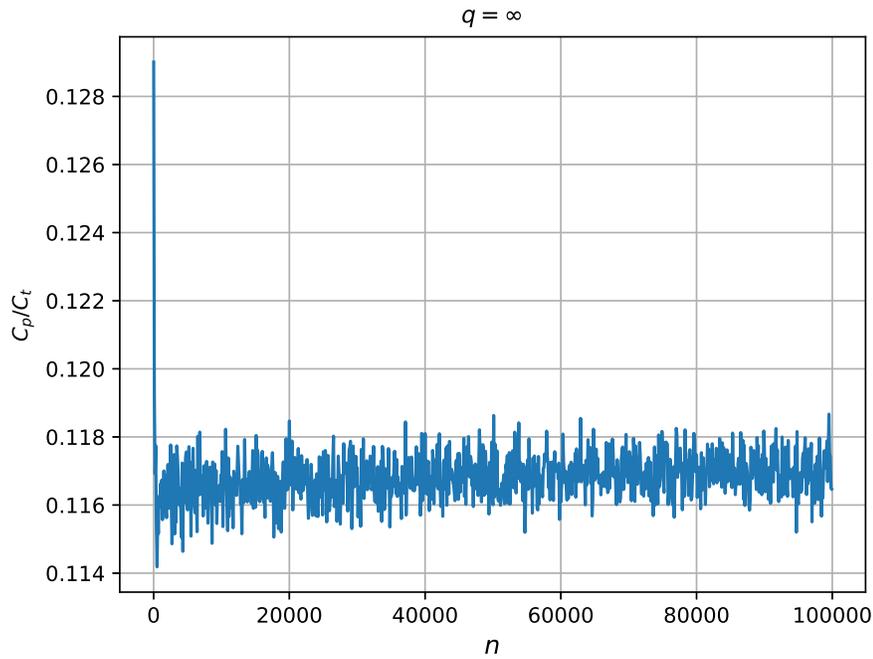}}
			\caption{Уточнение константы в оценке~\eqref{theorem1:expect_q_norm}, $n$~--- размерность пространства, $C_t = 16\ln n - 8$}
			\label{fig_exp1}
		\end{figure}

На рис.~\ref{fig_exp1} приведены результаты эксперимента
по оценке математического ожидания $\infty$-нормы векторов $e \in RS_2^n \left( 1 \right)$. По теореме~\ref{main_theorem}
при $q = \infty$ неравенство~\eqref{theorem1:expect_q_norm}
имеет вид
$$
        \E[||e||_\infty^2] \leqslant C_t n^{-1},
$$ 
где $C_t = 16\ln n - 8$. Эти же константы (назовём их в этом случае~$C_p$)
можно оценить практически, путём вычисления
$\E[||e||_\infty^2]$ методом Монте-Карло.
Это было сделано для $n$ от $10$ до $10^5$,
и на рис.~\ref{fig_exp1} приведено отношение
$C_p/C_t$ для разных $n$. Оказалось,
что отношение с ростом $n$ не меняется,
что значит, что теоретическая оценка 
с точностью до константы верна. 

\begin{figure}[ht]
			\center{\includegraphics[width=0.8
				\linewidth]{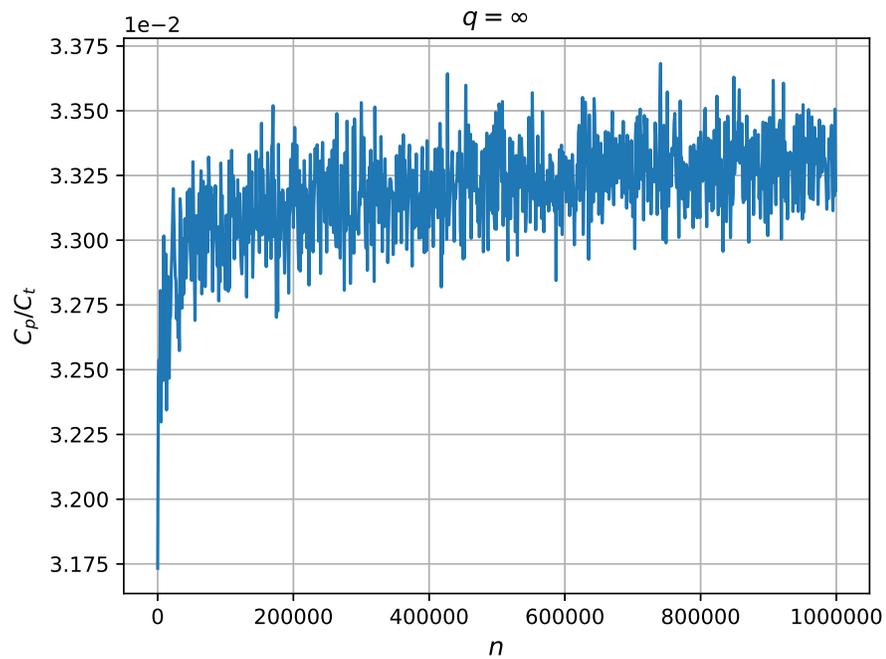}}
			\caption{Уточнение константы в оценке~\eqref{theorem1:expect_inner_product}, $n$~--- размерность пространства, $C_t = \sqrt{3}||s||_2^2\min\{2q-1, \, 32\ln n -8\}$}
			\label{fig_exp2}
		\end{figure}

Такие же эксперименты были проведены и для оценки константы
$$
C_t = \sqrt{3}||s||_2^2\min\{2q-1, \, 32\ln n -8\}
$$
в неравенстве~\eqref{theorem1:expect_inner_product}
при  $q = \infty$:
$$
\E[\langle s,\, e\rangle^2||e||_q^2] \leqslant C_t n^{-2}.
$$
Результаты экспериментов (см. рис.~\ref{fig_exp2}) также подтверждают, что теоретическая оценка $C_t$ 
с точностью до константы верна.

Код на языке Python всех вычислительных экспериментов
выложен на Github \cite{Code}.

\section{Благодарности}
Авторы выражают благодарность Павлу Евгеньевичу Двуреченскому за помощь в работе.


\clearpage


\begin{thebibliography}{99}

		\bibitem{ACDS_VGG}
		Е.\,A.~Воронцова, A.\,В.~Гасников, Э.\,А.~Горбунов.
		Ускоренные спуски по случайному
		направлению с неевклидовой прокс-структурой~//
		arXiv preprint arXiv:1710.00162
	
\bibitem{Arxiv_GLUF}
		A.~Gasnikov, A.~Lagunovskaya, I.~Usmanova, F.~Fedorenko.
		Gradient-free prox-methods with inexact oracle for stochastic convex optimization problems on a simplex~//
		arXiv preprint arXiv:1412.3890
		
\bibitem{AiT16_GLUF}
А.\,В.~Гасников, А.\,А.~Лагуновская, И.\,Н.~Усманова, Ф.\,A.~Федоренко.
Безградиентные прокc-методы с неточным оракулом для негладких задач выпуклой стохастической оптимизации на симплексе~//
Автомат. и телемех. 2016. 10. C.~57--77.

		\bibitem{Shamir_arx}
		O.~Shamir.
		An Optimal Algorithm for Bandit and Zero-Order Convex
Optimization with Two-Point Feedback~//
		arXiv preprint arXiv:1507.08752
		
\bibitem{Shamir_17}
O.~Shamir.
An Optimal Algorithm for Bandit and Zero-Order
Convex Optimization with Two-Point Feedback~//
Journal of Machine Learning Research.
2017. 18. P.~1--11.

\bibitem{Duchi_15}
J.\,C.~Duchi, M.\,I.~Jordan, M.\,J.~Wainwright, A.\,Wibisono.
Optimal Rates for Zero-Order Convex Optimization:
The Power of Two Function Evaluations~//
IEEE Transaction on Information Theory.
2015. 
Vol.~61.
Iss.~5. P.~2788--2806.

\bibitem{Blum_2016}
A.~Blum, J.~Hopcroft, R.~Kannan.
Foundations of Data Science.
Vorabversion eines Lehrbuchs, 2016.

\bibitem{Ball_1997}
K.~Ball.
An elementary introduction to modern convex geometry~//
Flavors of Geometry. Vol.~31.
Cambridge University Press, 1997.

\bibitem{bogolubsky2016learning}
L.~Bogolubsky, P.~Dvurechensky, A.~Gasnikov, G.~Gusev, Yu. Nesterov, A.~Raigorodskii, A.~Tikhonov, M.~Zhukovskii.
Learning Supervised PageRank with Gradient-Based and Gradient-Free Optimization Methods~// NIPS. 2016.


\bibitem{Boucheron_2013}
S.~Boucheron, G.~Lugosi, P.~Massart.
Concentration inequalities: A nonasymptotic theory of independence.
Oxford university press, 2013.

\bibitem{Milman_1986}
V.~Milman, G.~Schechtman.
Asymptotic Theory of Finite Dimensional Normed Spaces. (With an Appendix by M. Gromov).
Berlin, Springer-Verlag, 1986.

\bibitem{Zorich_2008}
В.~А.~Зорич.
Математический анализ в задачах естествознания.
МЦНМО, 2008.

\bibitem{BD}
		Баяндина~А.С., Гасников~А.В., Гулиев~Ф.Ш., Лагуновская~А.А.
		Безградиентные двухточечные методы решения задач стохастической негладкой выпуклой оптимизации при наличии малых шумов не случайной природы~//
		arXiv preprint arXiv:1701.03821


\bibitem{Poincare}
M.~Lifshits.
Lectures on Gaussian Processes.
Heidelberg Dordrecht London New York,
Springer,
2012.

\bibitem{Code}
Э.\,А.~Горбунов, Е.\,A.~Воронцова. 
Вычислительные эксперименты, иллюстрирующие явление концентрации равномерной меры на поверхности евклидовой сферы в малой окрестности экватора. URL: 
\url{https://github.com/evorontsova/Concentration-of-Measure/blob/master/Concentration%20of%20Measure.ipynb}
		







 





		\end{thebibliography}
\end{document}